\documentclass[12pt]{amsart}
\calclayout

\title[On the Schwarz Lemma for the Bergman metric]{On the Schwarz Lemma for Bergman metrics of bounded domains}

\author[Hoseob Seo]{Hoseob Seo}
\author[Sungmin Yoo]{Sungmin Yoo}
\author[Jihun Yum]{Jihun Yum}
\address{Research Institute of Mathematics, Seoul National University, 1 Gwanak-ro, Gwanak-gu, Seoul, 08826, Republic of Korea }
\email{hoseobseo@snu.ac.kr}
\address{Department of Mathematics, Incheon National University, (Songdo-dong) 119 Academy-ro Yeonsu-gu, Incheon, 22012, Republic of Korea}
\email{sunminyoo.math@inu.ac.kr}
\address{Department of Mathematics Education and Institute of Pure and Applied Mathematics, Jeonbuk National University, 567, Baekje-daero, Deokjin-gu, Jeonju-si, Jeonbuk-do, 54896, Republic of Korea}
\email{jihun0224@jbnu.ac.kr}
\date{\today}

\usepackage{amsthm,amssymb,latexsym}
\usepackage{comment}

\usepackage[abbrev]{amsrefs}
\usepackage{hyperref}

\usepackage[autostyle]{csquotes}
\usepackage{xcolor}
\usepackage[noadjust]{cite}

\usepackage{chemarrow}

\usepackage{kotex}
\usepackage{tikz-cd}

\newcommand{\RR}{\mathbb{R}}		
\newcommand{\CC}{\mathbb{C}} 		
\newcommand{\DD}{\Delta} 		    
\newcommand{\PP}{\mathbb{P}} 		
\newcommand{\BB}{\mathbb{B}} 		
\renewcommand{\O}{\Omega}			

\newcommand{\Prob}{\mathcal{P}}     
\newcommand{\Bergman}{\mathcal{B}}  
\newcommand{\Poisson}{P}            

\newcommand{\Exp}{\operatorname{E}}
 
\newcommand{\Var}{\operatorname{Var}} 

\newcommand{\norm}[1]{\left\|#1\right\|}                

\newcommand{\p}{\partial}

\newcommand{\ov}[1]{\overline{#1}}

\theoremstyle{plain}
\newtheorem{thm}{Theorem}[section]
\newtheorem{lem}[thm]{Lemma}
\newtheorem{prop}[thm]{Proposition}
\newtheorem{cor}[thm]{Corollary}
\newtheorem{ex}[thm]{Example}
\newtheorem{mainthm}{Theorem}

\theoremstyle{definition}
\newtheorem{defn}[thm]{Definition}

\theoremstyle{remark}
\newtheorem{rmk}[thm]{Remark}

\theoremstyle{property}

\numberwithin{equation}{section}

\begin{document}
	
\maketitle

\begin{abstract}
    We present a new Schwarz Lemma for bounded domains with Bergman metrics.
    The key ingredient of our proof is the Cauchy-Schwarz inequality from probability theory.
\end{abstract}

\section{Introduction}

    The classical Schwarz-Pick Lemma in one complex variable states that, for a holomorphic function $f: \DD \rightarrow \DD$ between unit discs $\DD \subset \CC$, 
    \begin{align*}
        \frac{|f'(z)|}{1-|f(z)|^2} \le \frac{1}{1-|z|^2}
    \end{align*}
    for all $z \in \DD$. Moreover, the equality holds at any point of $\DD$ if and only if $f$ is a biholomorphism. It is equivalent to say that geometrically $ f^* g_{\DD} \le g_{\DD} $, 
    where $g_{\DD}$ denotes the Poincar\'e metric (or the Bergman metric) of $\DD$. 
    This result has undergone several significant generalizations over the decades. Ahlfors(\cite{ahlfors1938extension}) first extended the lemma to Riemann surfaces by considering codomains with negative Gaussian curvature. This was later expanded to higher dimensions by Chern(\cite{chern1968holomorphic}) and Lu(\cite{lu1968holomorphic}), independently, who established the lemma for holomorphic maps from the unit ball to K\"ahler manifolds with negative holomorphic bisectional curvature. A landmark result was achieved by Yau(\cite{yau1978general}), who proved that for any holomorphic map $f: M_1 \to M_2$ between a complete K\"ahler manifold $(M_1, g_1)$ and a Hermitian manifold $(M_2, g_2)$, the inequality 
    $$f^*g_2 \leq \frac{K_1}{K_2}g_1$$
    holds, provided that the Ricci curvature of $M_1$ is bounded below by $K_1$ and the holomorphic bisectional curvature of $M_2$ is bounded above by $K_2 < 0$. Subsequently, Royden(\cite{royden1980ahlfors}) relaxed the conditions on the codomain $(M_2, g_2)$. He demonstrated that the inequality remains valid even without the holomorphic bisectional curvature condition, requiring only that the holomorphic sectional curvature of $g_2$ be bounded above by a negative constant. 
    Beyond its theoretical evolution, the Schwarz Lemma finds numerous applications across complex geometry. \\

    In this article, we introduce a new Schwarz Lemma for the bounded domains with the Bergman metrics. 
    The main theorem is the following. 

    \begin{mainthm}[Theorem~\ref{thm: Schwarz lemma for Bergman metric}] \label{mainthm: thm A}
        Let $(\O_1, g_{B_1})$ and $(\O_2, g_{B_2})$ be bounded domains in $\CC^n$ and $\CC^m$ with Bergman metrics, respectively.
        Suppose that there exists a constant $C > 0$ such that 
        \begin{align} \label{1.1}
            \sup_{w,\zeta \in \O_2} |\p_w \log P_2 (w, \zeta)|^2_{g_{B_2}} \le C,
        \end{align}
        where $P_2$ is the density function of the Bergman statistical model (Section~\ref{subsec: Bergman statistical model}) for $\O_2$. Then,
        for any holomorphic map $f:\O_1 \rightarrow \O_2$, 
        $$
            f^*g_{B_2} \le C g_{B_1}.
        $$
    \end{mainthm}

    A key advantage of Theorem~\ref{mainthm: thm A} over existing results by Yau and Royden is that it imposes no requirements on the domain $\Omega_1$ of the holomorphic map $f: \Omega_1 \rightarrow \Omega_2$; specifically, it requires neither completeness nor a lower curvature bound for $g_{B_1}$.
    As for the condition on the codomain $\Omega_2$, the relative merits of our condition (\ref{1.1}) compared to the negative upper bounded on curvature remain to be further explored. Nonetheless, we showed that our condition holds for any bounded homogeneous domain $\Omega_2$ (Corollary~\ref{cor: homogeneous Schwarz Lemma}). In fact, the condition (\ref{1.1}) is closely related to Gromov's K\"ahler hyperbolicity and the boundedness of the Bergman representative map (see Section~\ref{sec: gradient norm of logP}). 

    On the other hand, Suzuki(\cite{suzuki1985generalized}) utilized the Sibony's P-metric to show that the Schwarz Lemma for the Bergman metric holds for any holomorphic map $f: \DD \rightarrow \Omega_2$ under certain condition on $\O_2$. Interestingly, since this condition satisfies the requirements of Theorem~\ref{mainthm: thm A}, our theorem is a generalization of Suzuki's result (Corollary~\ref{cor: Suzuki thm}). \\

    The primary strength of this paper lies in our novel approach to proving the Schwarz Lemma, which is rooted in information geometry. While most generalizations of the Schwarz Lemma follow the fundamental framework established by Ahlfors--defining a function $u = f^*g_2 / g_1$ and applying the maximum principle--we instead employ the Cauchy-Schwarz inequality from probability theory.

    \begin{mainthm}[Theorem~\ref{thm: applying Cauchy-Schwarz inequality}, Theorem~\ref{thm: when the equality holds}] \label{mainthm: thm B}
        Let $(\O_1, g_{B_1})$ and $(\O_2, g_{B_2})$ be bounded domains in $\CC^n$ and $\CC^m$ with Bergman metrics, respectively.
        Then, for any $z \in \O_1$, $X \in T_z^{1,0}\O_1$ and a holomorphic map $f: \O_1 \rightarrow \O_2$,
        \begin{align*} 
            |f^*g_{B_2} (X,X)|^2 \le C_f(z,X) \, g_{B_1}(X,X) 
        \end{align*}
        holds, where $C_f(z,X) := \operatorname{Var}[\partial_X \log P_2(f(z), f(\cdot))] $. Moreover, the equality holds for all $z \in \O_1$ and $X \in T^{1,0}_z\O_1$ if and only if 
        $$f^* g_{B_2} = \lambda \, g_{B_1}$$ 
        for some constant $\lambda \ge 0$, i.e., $f$ is a $\lambda$--isometry. 
    \end{mainthm}

    As a direct consequence of the Cauchy-Schwarz inequality, Theorem \ref{mainthm: thm B} stands as a fundamental result with significant potential for further refinement. By imposing the additional condition (\ref{1.1}), we establish Theorem \ref{mainthm: thm A} as an application of Theorem \ref{mainthm: thm B}.


\vspace{5mm}

\section{Preliminaries}

\subsection{Bergman kernels and Bergman metrics}
    Let $\Omega$ be a bounded domain in $\CC^n$. Define the {\it Bergman space}, denoted by $A^2(\Omega)$, to be the set of all holomorphic functions on $\Omega$ that are square integrable with respect to the Lebesgue measure. We can equip $A^2(\Omega)$ with the Hermitian inner product: for any two elements $f$ and $g$ in $A^2(\Omega)$, $\langle f,g \rangle$ is defined by
    $$
        \langle f,g \rangle := \int_{\Omega} f(z) \overline{g(z)} dV,
    $$
    where $dV$ is the Lebesgue measure. With this inner product, the mean-value property for holomorphic functions and Montel's theorem tell us that $A^2(\Omega)$ is a separable Hilbert space.

    \begin{lem}
        For each $z \in \Omega$, the linear functional $\mathrm{eval}_z: A^2(\Omega) \to \CC$, defined by $f \longmapsto f(z)$, is continuous.
    \end{lem}

    Thanks to the Riesz representation theorem, there exists $K_z \in A^2(\Omega)$ such that
    \begin{align} \label{equ: reproducing property}
        f(z) = \mathrm{eval}_z (f) = \langle f, K_z \rangle = \int_{\Omega} f(\xi) \overline{K_z (\xi)} dV(\xi).
    \end{align}
    Write $\Bergman(z,\xi) := \overline{K_z(\xi)}$. This function on $\Omega \times \Omega$ is called the {\it Bergman kernel} and \eqref{equ: reproducing property} is called the reproducing property of the Bergman kernel.
    
    Let $\{ \phi_j \}_{j=0}^{\infty}$ be an orthonormal basis for $A^2(\Omega)$. Then, the Bergman kernel can be written as
    $$
        \Bergman(z,\xi) = \sum_{j=0}^{\infty} \phi_j(z)\overline{\phi_j(\xi)},
    $$
    which is independent of the choice of an orthonormal basis. This shows that $\Bergman(z,\xi)$ is holomorphic in $z$ while it is anti-holomorphic in $\xi$.
    
    Provided that the Bergman kernel is positive on the diagonal, we can define a hermitian form on the tangent space of $\Omega$ as follows: for two tangent vectors $X = (X_1,\ldots, X_n)$ and $Y=(Y_1,\ldots, Y_n) \in T_z^{1,0}\Omega$, 
    $$
        g_B (X,Y) = \sum_{j,k=1}^{n} \frac{\partial^2}{\partial z_j \partial \bar z_k} \log{\Bergman(z,z)}  X_j \overline{Y_k}.
    $$
    It is known that, for a bounded domain $\O \subset \CC^n$, $\Bergman(z,z) > 0$ for all $z \in \Omega$ and $g_B$ is positive definite everywhere. The K\"ahler metric $g_B$ is called the {\it Bergman metric} of $\O$. \\

    Fix $z_0 \in \Omega$ and $X = (X_1,\ldots, X_n) \in  T_{z_0}^{1,0}\Omega$. Let $\mathcal{H}_0$ be the kernel of $\mathrm{eval}_{z_0}$ and take $s_0 \in A^2(\Omega)$ to be a unit vector in the orthogonal complement of $\mathcal{H}_0$. On $\mathcal{H}_0$, define a continuous linear functional $\partial_X : \mathcal{H}_0 \to \CC$ by
    $$
        f \longmapsto (\partial_X f)(z_0) = \sum_{j=1}^n X_j \frac{\partial f}{\partial z_j} (z_0).
    $$
    Again, let $\mathcal{H}_1 \subset \mathcal{H}_0$ be the kernel of $\partial_X$ and let $s_1 \in \mathcal{H}_0$ be a unit vector in the orthogonal complement of $\mathcal{H}_1$ in $\mathcal{H}_0$. We extend $\{s_0,s_1\}$ to a complete system $\{s_0,s_1,\ldots\}$ for $A^2(\Omega)$ by adding any orthonormal basis for $\mathcal{H}_1$. With this basis, a direct computation yields
    \begin{align*}
        \Bergman(z_0,z_0) &= |s_0(z_0)|^2 > 0, \\
        g_B (X,X) &= \frac{|\partial_X s_1 (z_0)|^2}{|s_0(z_0)|^2} > 0.
    \end{align*}
    This argument works since $\Omega$ is a bounded domain and such a complete system is referred to as a Bergman's {\it special orthonormal basis with respect to $z_0 \in \O$ and $X \in T^{1,0}_{z_0} \O$}.

    \begin{lem}[cf. \cite{Wang15}] \label{lem: A^2 separates points}
        $A^2(\O)$ separates points, i.e., for all $z_1, z_2 \in \O$, there exists $f \in A^2(\O)$ such that $f(z_1)=0$ and $f(z_2) \neq 0$, if and only if $|\Bergman(z_1, z_2)|^2 \neq \Bergman(z_1,z_1) \Bergman(z_2,z_2).$
    \end{lem}
    \begin{proof}
        Consider an orthonormal basis $\{ \phi_j \}_{j=0}^{\infty}$ for $A^2(\O)$ satisfying $\phi_0 = f / \norm{f}$.
        Then $(\phi_0(z_1), \phi_1(z_1), \dots )$ and $(\phi_0(z_2), \phi_1(z_2), \dots )$ are linearly independent, which is equivalent to $|\Bergman(z_1, z_2)|^2 \neq \Bergman(z_1,z_1) \Bergman(z_2,z_2)$ by the Cauchy--Schwarz inequality.
    \end{proof}

\subsection{Bergman statistical models} \label{subsec: Bergman statistical model}
    For a bounded domain $\O \subset \CC^n$, let $\Prob(\O)$ be the set of probability distributions with positive density function on $\O$. That is, $\Prob(\O)$ is defined by
    $$
    \Prob(\O) = \left\{ \mu \in \mathfrak{M} (\O) : d\mu = \phi dV \text{ for some positive } \phi \in L^1(\O,dV)  \text{ and } \int_{\O}d\mu = 1\right\},
    $$
    where $\mathfrak{M}(\O)$ is the set of positive measures on $\O$ and $dV$ is the Lebesgue measure. Regarding $\Prob(\O)$ as a subset of the Banach manifold $\mathcal{S}(\O)$ of all signed measures on $\O$, the tangent space of $\Prob(\O)$ at $\mu$ is given by
    $$
    T_\mu \Prob(\O) := \left\{ \sigma \in \mathcal{S}(\O) \,:\, \int_{\O} d\sigma = 0 \text{ and } \sigma \ll \mu \right\}.
    $$
    Note that if we identify $T_{\mu}\Prob(\O)$ with the set of Radon-Nikodym derivatives of probability measures, then we have
    $$
    T_\mu \Prob(\O) \cong \left\{ \phi \in L^1 (\O, \mu) \,:\, \int_{\O} \phi \, d\mu = 0 \right\}.
    $$
    
    The {\it Fisher information metric} $g_F$ at $\mu \in \Prob(\O)$ is a quadratic form on $T_\mu \Prob(\O)$ defined by, for all $\sigma_1, \sigma_2 \in T_\mu \Prob(\O)$,
    $$
        (g_F)_\mu (\sigma_1, \sigma_2) := \int_{\O} \frac{d\sigma_1}{d\mu} \frac{d\sigma_2}{d\mu} d\mu =  \int_{\O} \phi \, \psi P dV,
    $$
    where $d\mu = PdV$, $d\sigma_1 = \phi d\mu$, and $d\sigma_2 = \psi d\mu$. \\

    Let $M$ be a smooth manifold and let $\Phi$ be a map from $M$ to $\Prob (\O)$. The triple $(M, \O, \Phi)$ is called a {\it statistical model} (or {\it statistical manifold}) (\cite{amari2000methods}, \cite{ay2015information}) if 
    \begin{enumerate}
        \item $\log{P(x,\xi)}$ is defined as a $C^1$ function in $x$ for almost all $\xi \in \O$, where $P(x,\xi)$ is the Radon-Nikodym derivative of $\Phi(x)$ with respect to $dV$, and
        \item for all continuous vector field $X$ on $M$, the function $\p_X \log{P(x,\cdot)}$ belongs to $L^2(\O, \Phi(x))$ and its $L^2$ norm varies continuously on $M$. 
    \end{enumerate}

    In this article, we consider the following statistical model.  
    For a bounded domain $\O$ in $\CC^n$,
    define the map $\Phi:\O \to \Prob(\O)$ by
    $$
        \Phi(z) := P(z,\xi)dV(\xi) := \frac{|\Bergman(z,\xi)|^2}{\Bergman(z,z)}dV(\xi).
    $$
    The triple $(\O, \O, \Phi)$ is called the {\it Bergman statistical model} for $\O$. 
    We refer to \cite{cho2023statistical} for fundamental properties of Bergman statistical models. The starting point for studying Bergman geometry from a statistical viewpoint is the discovery of the relationship between the Bergman metric of $\O$ and the Fisher information metric of $\Prob(\O)$.

    \begin{thm}[\cite{BurbeaJacobRaoCRadhakrishna84}, {\cite[Theorem~3.9]{cho2023statistical}}] \label{thm:[CY2023]Thm3.9}
        For a bounded domain $\O \subset \CC^n$, the map $\Phi$ is an embedding from $\O \to \Prob(\O)$. Moreover, we have
        $$
            g_B (X,Y) = (\Phi^*g_F)(X,\overline{Y})
        $$
        for any $X, Y \in T^{1,0}_z\O$, where $g_B$ is the Bergman metric of $\O$.
    \end{thm}

    The density function $P(z,\xi)$ for the Bergman statistical model is usually called the {\it Berezin kernel} or  the {\it Poisson-Bergman kernel}. Although it is real-valued, in contrast to the Bergman kernel, it also satisfies the following reproducing property.
    \begin{prop}[\cite{cho2023statistical}] \label{prop: Poisson reproducing}
		For $z \in \O$, assume that either
		\begin{equation*} 
			 f(\cdot)\Bergman(\cdot, z) \in A^2(\O) \quad \text{ or } \quad \overline{ f(\cdot) \Bergman(z, \cdot)} \in A^2(\O).
		\end{equation*}
		Then
		$$ f(z) = \int_{\O}  f(\xi) \Poisson(z,\xi) dV(\xi). $$
	\end{prop}

\subsection{Cauchy-Schwarz inequality} In this subsection, we will define the notions of expectation and (co)variance of measurable functions $Z$ and $W$. Although these notions originate from probability theory, a deep background in probability is not necessary to follow the material. Consider a statistical model $(M, \Omega, \Phi)$. Denote by $P(x, \cdot)$ the density function of $\Phi(x) \in \Prob(\Omega)$ for $x \in M$.

    \begin{defn} \label{def:Moments of random variables}
   Let $Z$ and $W$ be functions from $M \times \O$ to $\CC$ such that, for all $x \in M$, $Z(x,\cdot)$ and $W(x, \cdot)$ are measurable functions with respect to the Lebesgue measure. For each $x \in M$,
        \begin{enumerate}
            \item the {\it expectation} of $F$ is defined by
            $$
                \operatorname{E}[Z] := \int_{\O} Z(x,\xi) P(x,\xi) dV(\xi).
            $$
            \item The {\it covariance} of $Z$ and $W$ is defined by
            \begin{align*}
                \operatorname{Cov}[Z,W] 
                &:= \operatorname{E} [(Z-\operatorname {E} [Z]){\overline {(W-\operatorname {E} [W])}}] \\
                &= \operatorname {E} [Z \overline{W}] - \operatorname {E} [Z] \operatorname {E} [\overline{W}]. 
            \end{align*}
            \item The {\it variance} of $Z$ is defined by
            \begin{align*}
                \operatorname{Var}[Z]
                &:= \operatorname {Cov} [Z,Z] \\
                &= \operatorname {E}[|Z|^2]  - \left|\operatorname {E} [Z]\right|^{2}.
            \end{align*}
        \end{enumerate}
    \end{defn}

    Note that the expectation operator can be used to define a hermitian inner product. Indeed, for each point $x \in M$, let $L^2 (\Omega, \Phi(x))$ be the space of measurable functions on $\O$ that are $L^2$ with respect to the probability measure $\Phi(x) = P(x,\cdot)dV$. Then, the expectation and the variance of a function in $L^2(\O,\Phi(x))$ is well-defined. For two functions $Z(x,\cdot),W(x,\cdot) \in L^2(\O,\Phi(x))$, the expectation $\operatorname{E}[Z \overline{W}]$ given as in the above definition defines a hermitian inner product. 
    The following lemma is crucial in this paper. 

    \begin{lem}[Cauchy-Schwarz inequality] \label{lem: Cauchy-Schwarz for variables}
    For $Z(x,\cdot)$ and $W(x,\cdot)$ in $L^2(\O, \Phi(x))$,
        \begin{align*}
            \left|\operatorname{Cov}[Z,W]\right|^2 \le \operatorname{Var}[Z] \operatorname{Var}[W].
        \end{align*}
        The equality holds if and only if $Z(x,\xi)-\operatorname{E}[Z] = \lambda(x) (W(x,\xi) - \operatorname{E}[W])$ for some $\lambda(x) \in \CC$.
    \end{lem}

\vspace{5mm}


\section{Schwarz lemmas for Bergman metrics}
In this section, we present our main results and their proofs.
The main ingredients are the reproducing property (Proposition~\ref{prop: Poisson reproducing}) and a special orthonormal basis for $A^2(\O)$.

For bounded domains $\O_1$ and $\O_2$, let $P_1$ and $P_2$ be the density functions of the Bergman statistical models for $\O_1$ and $\O_2$, respectively. 
For a vector $X \in T^{1,0}\O_1$, we denote by $\p_X = \sum_{j=1}^n X_j \frac{\p}{\p z_j}$ and $\p_{\ov{X}} = \sum_{j=1}^n \ov{X}_j \frac{\p}{\p \ov{z}_j}$.

    \begin{lem} \label{lem:proof in main thm}
        Let $\Omega_1 \subset \CC^n$ and $\Omega_2 \subset \CC^m$ be bounded domains and 
        $f: \O_1 \rightarrow \O_2$ be a holomorphic map.
        Then, for fixed $z \in \O_1$ and $X \in T_z^{1,0}\O_1$,
        \begin{align} \label{equ: E[dlogP]=0}
            \Exp[\p_X \log P_1(z,\cdot)]
            = \int_{\O_1} \partial_X P_1(z,\xi) dV(\xi) 
            = 0.
        \end{align}
        Moreover, under the assumptions that $\operatorname{Var}[\partial_X \log P_2(f(z), f(\cdot))] < \infty$ and $df(X) \neq 0$,
        \begin{align*}
            \int_{\O_1} \partial_{\overline{X}} \log \Bergman_2(f(\xi), f(z)) \, \partial_X P_1(z,\xi) dV(\xi) = f^*g_{B_2}(X,X).
        \end{align*}
    \end{lem}
    \begin{proof}
        Let $\{ s_j \}_{j=0}^{\infty}$ be a special orthonormal basis for $A^2(\O_1)$ with respect to $z \in \O_1$ and $X \in T_z^{1,0} \O_1$. 
        Also, let $\{ \sigma_j \}_{j=0}^{\infty}$ be a special orthonormal basis for $A^2(\O_2)$ with respect to $f(z) \in \O_2$ and $df(X) \in T_{f(z)}^{1,0} \O_2$. Then
        \begin{align*}
            \int_{\O_1} \partial_X P_1(z,\xi) dV(\xi)
            = \int_{\O_1} \frac{\partial_X s_1(z)}{s_0(z)} s_0(\xi)  \overline{s_1(\xi)} dV(\xi)
            = \frac{\partial_X s_1(z)}{s_0(z)} \int_{\O_1} s_0(\xi)  \overline{s_1(\xi)} dV(\xi)
            = 0.\\
        \end{align*}

        For the second equation, we first analyze the assumption. If we express it in terms of special orthonormal bases, then 
        \begin{align*}
            & \int_{\O_1} \left| \partial_X \log P_2(f(z), f(\xi)) \right|^2 P_1(z,\xi) dV(\xi) \\
            =&  \frac{|\partial_X \sigma_0(f(z))|^2}{|\sigma_0(f(z))|^2} 
            \int_{\O_1} \frac{|\sigma_1(f(\xi))|^2}{|\sigma_0(f(\xi))|^2} |s_0(\xi)|^2 dV(\xi) < \infty.
        \end{align*}
        The function $S(\xi) := \frac{\sigma_1(f(\xi))}{\sigma_0(f(\xi))} s_0(\xi)$ is holomorphic on $\O_1 \setminus \{ \xi \in \O_1 : (\sigma_0 \circ f)(\xi) = 0 \}$ and $S \in L^2(\O_1)$. By the removable singularity theorem, $S \in A^2(\O_1)$, and hence it can be written as $S(\xi) = \sum_{j=0}^{\infty} a_j s_j(\xi)$ for some constant $a_j \in \CC$.
        Then
        \begin{align*}
            a_0 s_0(z) = S(z) = \frac{\sigma_1(f(z))}{\sigma_0(f(z))} s_0(z) = 0
        \end{align*}
        implies $a_0=0$, 
        and
        \begin{align*} 
            a_1 \partial_X s_1(z)  
            &= \partial_X S(z) 
            = \partial_X \sigma_1(f(z)) \cdot \frac{s_0(z)}{\sigma_0(f(z))} + \sigma_1(f(z)) \cdot \partial_X \left( \frac{s_0(z)}{\sigma_0(f(z))} \right)
        \end{align*}
        yields 
        \begin{align*}
            a_1 =  \frac{\partial_X \sigma_1(f(z))}{\sigma_0(f(z))} \cdot \frac{s_0(z)}{\partial_X s_1(z)}.
        \end{align*}
        Now 
        \begin{align*}
            \partial_{\overline{X}} \log \Bergman_2(f(\xi), f(z)) 
            &=  \frac{\overline{\partial_X \sigma_0(f(z))}}{\overline{\sigma_0(f(z))}}
            + \frac{\overline{\partial_X \sigma_1(f(z))}}{\overline{\sigma_0(f(z))}}
            \cdot \frac{\sigma_1(f(\xi))}{\sigma_0(f(\xi))},  \\
            \partial_X P_1(z,\xi)
            &= \frac{\partial_X s_1(z)}{s_0(z)} s_0(\xi) \overline{s_1(\xi)}
        \end{align*}
        imply that
        \begin{align*}
             & \int_{\O_1} \partial_{\overline{X}} \log \Bergman_2(f(\xi), f(z)) \, \partial_X P_1(z,\xi) dV(\xi) \\
             =& \frac{\overline{\partial_X \sigma_0(f(z))}}{\overline{\sigma_0(f(z))}} \frac{\partial_X s_1(z)}{s_0(z)} \int_{\O_1} s_0(\xi) \overline{s_1(\xi)} dV(\xi) 
             + \frac{\overline{\partial_X \sigma_1(f(z))}}{\overline{\sigma_0(f(z))}} \frac{\partial_X s_1(z)}{s_0(z)} \int_{\O_1} S(\xi) \overline{s_1(\xi)} dV(\xi) \\
             =& \frac{\overline{\partial_X \sigma_1(f(z))}}{\overline{\sigma_0(f(z))}} \frac{\partial_X s_1(z)}{s_0(z)} a_1 
             =  \frac{|\partial_X \sigma_1(f(z))|^2}{|\sigma_0(f(z))|^2} 
             = f^*g_{B_2}(X,X).
        \end{align*}
    \end{proof}

    \begin{rmk}
        The equation (\ref{equ: E[dlogP]=0}) in Lemma~\ref{lem:proof in main thm} appears to follow naturally from the fact that
        \begin{align*}
            \int_{\Omega_1} \partial_X P_1(z,\xi) dV(\xi) = \partial_X \int_{\Omega_1}  P_1(z,\xi) dV(\xi) = 0.
        \end{align*}
        However, to rigorously complete the argument, one must justify that the integration and differentiation are interchangeable. This is not trivial, as the off-diagonal Bergman kernel $\Bergman(z, \cdot)$ is not bounded in general (\cite{kerzman1971bergman}).
    \end{rmk}

    \begin{thm} \label{thm: applying Cauchy-Schwarz inequality}
        Let $\Omega_1 \subset \CC^n$ and $\Omega_2 \subset \CC^m$ be bounded domains and 
        $f: \O_1 \rightarrow \O_2$ be a holomorphic map.
        Then, for each $z \in \O_1$ and $X \in T_z^{1,0}\O_1$,
        \begin{align} \label{main inequality}
            |f^*g_{B_2} (X,X)|^2 \le \operatorname{Var}[\partial_X \log P_2(f(z), f(\cdot))] \, g_{B_1}(X,X).
        \end{align}
    \end{thm}
    \begin{proof}
        We may assume that $\operatorname{Var}[\partial_X \log P_2(f(z), f(\cdot))] < \infty$ and $df(X) \neq 0$; otherwise, the inequality always holds.
        Let $Z = \p_X \log{P_1(z,\xi)}$ and $W = \p_X \log{P_2 (f(z),f(\xi))}$. Note that, from $\Exp[Z]=0$ (by (\ref{equ: E[dlogP]=0})) and Theorem~\ref{thm:[CY2023]Thm3.9}, we have
        \begin{align*}
            \operatorname{Var}\left[ Z \right] 
            = \int_{\O_1} |\p_X \log{P_1(z,\xi)}|^2  P_1 dV(\xi)
            = g_{B_1}(X,X).
        \end{align*}
        Moreover, by Lemma~\ref{lem:proof in main thm},
        \begin{align*}
            \operatorname{Cov}\left[ Z, W  \right]
            =& \int_{\O_1} \partial_{\overline{X}} \log P_2(f(z), f(\xi)) \, \partial_X P_1 dV(\xi)  \\
            =& \int_{\O_1} \partial_{\overline{X}} \log \Bergman_2(f(\xi), f(z)) \, \partial_X P_1 dV(\xi)
            -\int_{\O_1} \partial_{\overline{X}} \log \Bergman_2(f(z), f(z)) \, \partial_X P_1 dV(\xi)  \\
            =& f^*g_{B_2} (X,X). 
        \end{align*}
        Applying the Cauchy-Schwarz inequality (Lemma~\ref{lem: Cauchy-Schwarz for variables}) for $Z$ and $W$ gives the desired inequality.      
    \end{proof}

    \begin{thm} \label{thm: when the equality holds}
        The following are equivalent.
        \begin{enumerate}
            \item[(a)] The equality holds in (\ref{main inequality}) for all $z \in \O_1$ and $X \in T^{1,0}_z\O_1$.
            \item[(b)] For all $z, \xi \in \Omega_1$ and $X \in T^{1,0}_z\O_1$ with $P_1(z,\xi) \neq 0$ and $P_2(f(z), f(\xi)) \neq 0$,
            $$\p_{X} \log{P_2}(f(z),f(\xi)) = \lambda \, \p_{X} \log{P_1}(z,\xi)$$ 
            for some constant $\lambda \ge 0$.
            \item[(c)] For all $z \in \O_1$ and $X \in T^{1,0}_z\O_1$,
            $$f^* g_{B_2}(X,X) = \lambda \, g_{B_1}(X,X)$$ 
            for some constant $\lambda \ge 0$, i.e., $f$ is a $\lambda$--isometry.
            \item[(d)] For all $z, \xi \in \Omega_1$,
            $$\frac{|\Bergman_2(f(z),f(\xi))|^2}{\Bergman_2(f(z),f(z))\Bergman_2(f(\xi),f(\xi))} = \left(\frac{|\Bergman_1(z,\xi)|^2}{\Bergman_1(z,z) \Bergman_1(\xi,\xi)}\right)^{\lambda}$$ 
            for some constant $\lambda \ge 0$.
        \end{enumerate}
        When one of above statements holds, 
        \begin{enumerate}
        	\item $\lambda = 0 $ if and only if $f$ is a constant map, and
        	\item $\lambda \neq 0$ if and only if $f$ is injective.
        \end{enumerate}
    \end{thm}
    \begin{proof}
        We prove the equivalence of four statements by showing $(a) \Leftrightarrow (b)$ and $(b) \Rightarrow (c) \Rightarrow (d) \Rightarrow (b)$. \\

        \noindent
        $(a) \Leftrightarrow (b):$
        Assume that $|f^*g_{B_2}(X,X)|^2 = g_{B_1}(X,X) \, \operatorname{Var}[W]$ holds for all $z \in \Omega_1$ and $X \in T^{1,0}_z\O_1$ with $\p_X = \sum_{j=1}^n X_j \frac{\p}{\p z_j}$. From the Cauchy-Schwarz inequality (Lemma~\ref{lem: Cauchy-Schwarz for variables}), we know that
        \begin{equation} \label{CSequality}
        	\p_{X} \log{P_2}(f(z),f(\xi)) = \lambda(z, X_1, \dots, X_n) \, \p_{X} \log{P_1}(z,\xi)
        \end{equation}
        holds for all $z, \xi \in \Omega_1$ and $X \in T^{1,0}_z\O_1$ such that $P_1(z,\xi) \neq 0$ and $P_2(f(z), f(\xi)) \neq 0$, 
        where $\lambda: \O_1 \times \CC^n \rightarrow \CC$. Since $\lambda(z,X) = \lambda(z, \alpha X)$ for any $\alpha \in \CC^*$, one has $\lambda: \Omega_1 \times \CC\PP^{n-1} \to \CC$. 
        
        Now, differentiating (\ref{CSequality}) by $\frac{\p}{\p \overline{X}_j} \in T^{0,1}_{X} (T^{1,0}_z \O_1)$ yields
        \begin{align*}
            \frac{\p \lambda}{\p \overline{X}_j} \cdot \p_X \log P_1(z,\xi) = 0 
        \end{align*} 
        for all $j=1, \dots, n$.
        Suppose that there exist $z \in \O_1$ and $X \in T^{1,0}_z \O_1$ such that $\p_X \log P_1(z, \cdot) = 0$. Then this contradicts to the map $\Phi: \O_1 \rightarrow \Prob(\O_1)$ is an immersion by Theorem~\ref{thm:[CY2023]Thm3.9}. Therefore, for each $z \in \O_1$ and $X \in T^{1,0}_z \O_1$, there exists $\xi \in \O_1$ such that $\p_X \log P_1(z, \xi) \neq 0$, which implies that $\frac{\p \lambda}{\p \overline{X}_j} = 0$ for all $j=1, \dots, n$. Hence, $\lambda(z, \cdot)$ is a holomorphic function on $\CC\PP^{n-1}$ and a constant function. Write $\lambda(z,X) = \lambda(z)$.  
        On the other hand, differentiating (\ref{CSequality}) by $\sum_{j=1}^n X_j \left. \frac{\p}{\p \xi_j} \right|_{\xi=z}$ shows that
        \begin{equation} \label{metricProportional}
        	f^*g_{B_2}(X,X) = \lambda(z) g_{B_1}(X,X)
        \end{equation}
        holds for all $X \in \CC^n$ and thus $\lambda(z) \in [0,\infty)$. By taking $\p_{\overline{X}} = \sum_{j=1}^n \overline{X}_j \frac{\p}{\p \overline{z}_j} $ on both sides of (\ref{CSequality}), one obtains
        $$
        -f^*g_2(X,X) = \p_{\overline{X}} \lambda(z) \, \p_{X} \log{P_1}(z,\xi) - \lambda(z) g_1(X,X).
        $$
        Combining it with (\ref{metricProportional}) yields that $\lambda(z)$ is holomorphic in $z$. Since the image lies on the real line, $\lambda(z)$ must be a constant. 
        The converse implication is immediate because  $\p_{X} \log{P_2}(f(z),f(\cdot))$  and $\p_{X} \log{P_1}(z, \cdot)$ are linearly dependent.  \\

        \noindent
        $(b) \Rightarrow (c):$
        It follows by taking $\p_{\overline{X}} = \sum_{j=1}^n \overline{X}_j \frac{\p}{\p \overline{z}_j} $ on both sides of
        $$\p_{X} \log{P_2}(f(z),f(\xi)) = \lambda \, \p_{X} \log{P_1}(z,\xi).$$ 

        \noindent
        $(c) \Rightarrow (d):$

        From the condition $f^* g_{B2} = \lambda g_{B1}$, we have
        $$ \partial \overline{\partial} \log \Bergman_2(f(z),f(z)) 
        - \lambda \partial \overline{\partial} \log \Bergman_1(z,z) = 0 $$
        for all $z \in \O_1$. 
        Then there exist a (simply-connected) open neighborhood $U \subset \O_1$ and a holomorphic function $\varphi$ on $U$ such that
        \begin{equation} \label{3.6}
            \log \Bergman_2(f(z),f(z)) 
            - \lambda \log \Bergman_1(z,z) = \varphi(z) + \overline{\varphi(z)}
        \end{equation}
        for all $z \in U$ and hence 
        $$ \log \Bergman_2(f(z),f(\xi)) 
        - \lambda \log \Bergman_1(z,\xi) = \varphi(z) + \overline{\varphi(\xi)}
        $$
        for all $z, \xi \in U$ (by shrinking $U$ if necessary).
        Furthermore, 
        \begin{equation} \label{1.2}
            \log \Poisson_2(f(z),f(\xi)) 
            - \lambda \log \Poisson_1(z,\xi) = \varphi(\xi) + \overline{\varphi(\xi)}
        \end{equation} 
        for all $z, \xi \in U$.
        Denote the left side of (\ref{3.6}) (after replacing the $z$ variable by $\xi$) and the left side of (\ref{1.2}) by $A(\xi)$ and $B(z,\xi)$, respectively. Then both $\exp (A(\xi))$ and $\exp (B(z,\xi))$ are well-defined on the whole domain $\O_1 \times \O_1$.
        Since they are real-analytic functions and coincide on $U \times U$, we have 
        \begin{align*}
            \frac{|\Bergman_2(f(z),f(\xi))|^2}{\Bergman_2(f(z),f(z))\Bergman_2(f(\xi),f(\xi))} = \left(\frac{|\Bergman_1(z,\xi)|^2}{\Bergman_1(z,z) \Bergman_1(\xi,\xi)}\right)^{\lambda}
        \end{align*}
        for all $z, \xi \in \O_1$. \\
        
        \noindent
        $(d) \Rightarrow (b):$
        It follows by taking $\p_{X} = \sum_{j=1}^n X_j \frac{\p}{\p z_j} $ on both sides of
        \begin{align*}
            \log \left( \frac{|\Bergman_2(f(z),f(\xi))|^2}{\Bergman_2(f(z),f(z))\Bergman_2(f(\xi),f(\xi))} \right) = \lambda \log \left( \frac{|\Bergman_1(z,\xi)|^2}{\Bergman_1(z,z) \Bergman_1(\xi,\xi)} \right). \\
        \end{align*}

        Now, suppose that one of above statements holds. 
        First, from the condition $(c)$, $\lambda = 0 $ if and only if $f$ is a constant map because $g_{B_1}$ and $g_{B_2}$ are positive definite.
        Second, if $\lambda \neq 0$ and $f(z_1)=f(z_2)$ for some $z_1, z_2 \in \O_1$, the condition $(d)$ implies that 
        $$ 1 = \left(\frac{|\Bergman_1(z_1, z_2)|^2}{\Bergman_1(z_1,z_1) \Bergman_1(z_2, z_2)}\right)^{\lambda}. $$
        Hence, $z_1=z_2$ by Lemma~\ref{lem: A^2 separates points}. Note that, for a bounded domain $\O_1 \in \CC^n$, $A^2(\O_1)$ always separates points. 
        Conversely, for the sake of contradiction, assume that $\lambda = 0$. Then, for $z_1 \neq z_2 \in \O_1$, the condition $(d)$ says that 
        $$ \frac{|\Bergman_2(f(z_1),f(z_2))|^2}{\Bergman_2(f(z_1),f(z_1))\Bergman_2(f(z_2),f(z_2))} = 1 $$
        This is impossible because $f(z_1) \neq f(z_2)$.
    \end{proof}

    \begin{rmk}
        Theorem~\ref{thm: when the equality holds} says that if $f$ is a biholomorphism, then the equality holds in (\ref{main inequality}).
        Conversely, if the equality holds in (\ref{main inequality}) with a non-constant holomorphic map $f$, then $f$ is injective. 
        In fact, this result is optimal in the following sense: consider $f:  \DD \rightarrow \DD^2$ defined by $f(z) = (z,0)$, where $\DD$ denotes the unit disc in $\CC$. Then $f^* g_{\DD^2} = g_{\DD}$.
        When $n=m$, with additional conditions that $g_{B_1}$ and $g_{B_2}$ are complete, one can show that the equality holds in (\ref{main inequality}) implies $f$ is also surjective by \cite[Theorem 2.3.2]{mok2011geometry}, hence $f$ is a biholomorphism.
        For further details on holomorphic $\lambda$--isometries with respect to the Bergman metric, we refer the reader to \cite{mok2011geometry}, \cite{mok2012extension}, and \cite{yum2025bergman}.
    \end{rmk}

    Now we prove Theorem~\ref{mainthm: thm A}.
    Consider bounded domains $(\O_1, g_{B_1})$ and $(\O_2, g_{B_2})$ in $\CC^n$ and $\CC^m$ with Bergman metrics, respectively.
    Recall that $P_1$ and $P_2$ are the density functions of the Bergman statistical models for $\O_1$ and $\O_2$.
    Let $f:\O_1 \rightarrow \O_2$ be a holomorphic map.
    
     \begin{lem} \label{lem: Var[dlogP_2]}
        For fixed $z \in \O_1$ and $X \in T^{1,0}_z \O_1$, assume that $\operatorname{Var}[W] < \infty$, where $W =  \partial_{X} \log P_2(f(z), f(\xi))$.
        Then $\Exp[W] = 0$.
    \end{lem}
    \begin{proof}
        We have 
        \begin{align*}
            \operatorname{E}[W] 
            =& \int_{\O_1} \partial_{X} \log P_2(f(z), f(\xi)) P_1 dV(\xi) \\
            =& \int_{\O_1} \partial_{X} \log \Bergman_2(f(z), f(\xi))  P_1 dV(\xi)
            -\int_{\O_1} \partial_{X} \log \Bergman_2(f(z), f(z))  P_1 dV(\xi)  \\
            =& \int_{\O_1} \partial_{X} \log \Bergman_2(f(z), f(\xi))  P_1 dV(\xi)
            - \partial_{X} \log \Bergman_2(f(z), f(z)) = 0
        \end{align*}
        by Proposition~\ref{prop: Poisson reproducing} provided that $ \partial_{X} \log \Bergman_2(f(z), f(\cdot)) \Bergman_1(z, \cdot) \in L^2(\O_1)$, which is equivalent to $\operatorname{Var}[W] < \infty$.
    \end{proof}

    \begin{thm} \label{thm: Schwarz lemma for Bergman metric}
        Suppose that there exists a constant $C > 0$ such that 
        $$ \sup_{w,\zeta \in \O_2} |\p_w \log P_2 (w, \zeta)|^2_{g_{B_2}} \le C, $$
        where $\partial_w \log P_2$ is regarded as a $1$-form with respect to the $w$ variable. 
        Then
        for any holomorphic map $f:\O_1 \rightarrow \O_2$, 
        $$
            f^*g_{B_2} \le C g_{B_1}.
        $$
    \end{thm}
    \begin{proof} 
        Firstly, we show that $\operatorname{Var}[W] < \infty$ for fixed $z \in \O_1$ and $X \in T^{1,0}_z \O_1$ from the assumption, where $W =  \partial_{X} \log P_2(f(z), f(\xi))$. 
        Since the conclusion holds when $f^*g_{B_2}(X,X) = 0$, we may assume that $f^*g_{B_2}(X,X) > 0$.
        Then we have
        \begin{align} \label{3.7}
            \frac{\Exp[|W|^2]}{f^*g_{B_2}(X,X)} 
            & = \int_{\O_1} \frac{|\partial_X \log P_2(f(z),f(\xi))|^2}{f^*g_{B_2}(X,X)} P_1(z,\xi) dV(\xi) \\
            & = \int_{f(\O_1)} \frac{|\partial_{X} \log P_2(f(z),\zeta)|^2}{f^*g_{B_2}(X,X)} \kappa (P_1(z,\xi) dV(\xi))(\zeta) \nonumber \\
            & = \int_{f(\O_1)} \frac{|\partial_{df(X)} \log P_2(w,\zeta)|^2}{g_{B_2}(df(X),df(X))} \kappa (P_1(z,\xi) dV(\xi))(\zeta) \nonumber \\ 
            & \le \int_{f(\O_1)} |\partial_w \log P_2 (w, \zeta)|^2_{g_{B_2}} \kappa (P_1(z,\xi) dV(\xi))(\zeta)  \nonumber \\
            & \le C \int_{f(\O_1)} \kappa (P_1(z,\xi) dV(\xi))(\zeta)  \nonumber \\
            & = C,  \nonumber
        \end{align}
        where $\kappa$ is the measure push-forward of $f$, which is defined by $\kappa(\mu)(A) = \mu(f^{-1}(A))$  for a Borel set $A \subset \O_2$ and $\mu \in \Prob(\O_1)$. Here we used the fact that 
        \begin{align*}
            \sup_{Y \in T^{1,0}\O_2} \frac{|\partial_{Y} \log P_2(w,\zeta)|^2}{g_{B_2}(Y,Y)}
            = |\partial_w \log P_2 (w, \zeta)|^2_{g_{B_2}}.
        \end{align*}
        Therefore, $\operatorname{Var}[W] = \operatorname{E}[|W|^2] - \left|\operatorname{E}[W]\right|^2 < \infty$ and hence $\operatorname{E}[W] = 0$ by Lemma~\ref{lem: Var[dlogP_2]}.

        Now, by the same argument as (\ref{3.7}), we conclude that 
        \begin{align*}
            \frac{\Var[|W|^2]}{f^*g_{B_2}(X,X)} 
            = \frac{\Exp[|W|^2]}{f^*g_{B_2}(X,X)} 
            \le C,
        \end{align*}
        which completes the proof together with Theorem~\ref{thm: applying Cauchy-Schwarz inequality}.
    \end{proof}

    Theorem~\ref{thm: Schwarz lemma for Bergman metric} is a generalization of the Schwarz Lemma for the Bergman metric by Suzuki.
    \begin{cor}[cf. \cite{suzuki1985generalized}]  \label{cor: Suzuki thm}
        Let $(\O_1, g_{B_1})$ and $(\O_2, g_{B_2})$ be bounded domains in $\CC^n$ and $\CC^m$ with Bergman metrics, respectively.
        Suppose that there exists a constant $\alpha > 0$ such that
        \begin{align*}
            \alpha |\partial_{Y} \log P_2(w,\zeta)|^2 
            \le g_{B_2}(Y,Y)  \left( 1 - \left( \frac{|\Bergman_2(w,\zeta)|^2}{\Bergman_2(w,w) \Bergman_2(\zeta,\zeta)} \right)^{\alpha}   \right)
        \end{align*}
        for all $w, \zeta \in \O_2$ and $Y \in T^{1,0}_w \O_2$. 
        Then for any holomorphic map $f:\O_1 \rightarrow \O_2$, 
        $$
            f^*g_{B_2} \le \frac{1}{\alpha} g_{B_1} .
        $$
    \end{cor}
    \begin{proof}
        It follows from Theorem~\ref{thm: Schwarz lemma for Bergman metric} and the assumption, which implies that
        \begin{align*}
            \sup_{w,\zeta \in \O_2} |\p_w \log P_2 (w, \zeta)|^2_{g_{B_2}} \le \frac{1}{\alpha}.
        \end{align*}
    \end{proof}
    \begin{rmk}
        Suzuki(\cite{suzuki1985generalized}) proved Corollary~\ref{cor: Suzuki thm} for a holomorphic map $f: \DD \rightarrow \O_2$ where $\DD$ is the unit disc in $\CC$, providing the sharper estimate: $f^*g_{B_2} \le \frac{1}{2 \alpha} g_{B_1}$.
        While Suzuki obtained a sharper constant, we establish the Schwarz Lemma in a far more general setting, requiring no assumptions on the domain $\O_1$.
    \end{rmk}

\vspace{5mm}


\section{Gradient norm of $\log P$} \label{sec: gradient norm of logP}

\subsection{Boundedness of the gradient norm}
In this subsection, we discuss the condition:
\begin{equation}\label{gradient norm bound}
    \sup_{w,\zeta \in \O_2} | \partial_w \log P_2(w,\zeta)|_{g_{B_2}}^2 < \infty,
\end{equation}
in Theorem \ref{mainthm: thm A} (Theorem \ref{thm: Schwarz lemma for Bergman metric}),
and show that this condition holds for all bounded homogeneous domains.

In \cite{gromov1991kahler}, Gromov introduced the notion of {\it K\"ahler hyperbolicity}, the existence of a $d$-bounded K\"ahler form, for vanishing
theorems concerning the $L^2$-cohomology of complete K\"ahler manifolds.
Recall that for a complete K\"ahler manifold $(M,\omega)$, we say that $\omega$ is {\it $d$-bounded} if there exists a $1$-form $\eta$ on $M$ with $d\eta = \omega$ satisfying
    $$
        \Vert \eta \Vert_{\infty} = \sup_{M} |\eta|_\omega < \infty.
    $$

For a bounded pseudoconvex domain $\Omega\subset\mathbb{C}^n$, it is natural to consider the K\"ahler hyperbolicity of the Bergman metric
$$
\omega_{B} =i\partial\overline{\partial}\log\Bergman(z,z)=d(-i\partial\log\Bergman(z,z)).
$$
In \cite{donnelly1997, donnelly1994}, Donnelly proved that the Bergman metrics of strongly pseudoconvex domains and bounded homogeneous domains are $d$-bounded. 
In these cases, in fact, he showed that the boundedness of the {\it gradient norm} of $\log\Bergman(z,z)$:
$$
\sup_{z\in\Omega}|\partial\log \Bergman(z,z)|^2_{g_{B}}<\infty.
$$
    
On the other hand, we have other potential functions for the Bergman metric, comes from the Berezin kernels. 
More precisely, for each fixed $\zeta\in\Omega_2$, we have a global potential function $- \log P_2(\cdot,\zeta)$ of the Bergman metric of $\Omega_2$ and a $1$-form $i\partial\log P_2(\cdot,\zeta)$ satisfying
$$
\omega_{B_2} =i\partial\overline{\partial}(-\log P_2(\cdot,\zeta))=d(i\partial\log P_2(\cdot,\zeta)),
$$
where $\omega_{B_2}$ is the associated K\"ahler form of ${g_{B_2}}$.
In this context, the Bergman metric is $d$-bounded provided that for some $\zeta\in\Omega_2$, the gradient norm of $\log P_2(\cdot,\zeta)$ is bounded:
$$
\sup_{w\in\Omega_2}|\partial\log P_2(w,\zeta)|^2_{g_{B_2}} \le C_\zeta<\infty.
$$
Therefore, the condition (\ref{gradient norm bound}) means that the existence of uniform upper bound for the gradient norms of the functions $\log P_2(\cdot,\zeta)$ for all $\zeta\in\Omega_2$. 

Although this condition may seem restrictive, we will show that it holds at least for all bounded homogeneous domains.
In order to establish this, we begin by investigating the relationship between the gradient norm of $\log P_2(\cdot,\zeta)$ and the Bergman representative map.
Recall
\begin{defn}\label{rep}
Let $\O$ be a bounded domain in $\mathbb{C}^n$. 
For a fixed point $z_0\in\O$, the {\it Bergman representative map} at $z_0$ is defined by 
$$
\hbox{rep}_{z_0}(\xi)=(b_1(\xi),\ldots,b_n(\xi)),
$$
and
$$
b_j(\xi):=h^{\overline{k}j}(z_0) \left\{ \left. \frac{\partial}{\partial \overline{z}_k }\right|_{z=z_0}\log \Bergman(\xi,z)- \left.\frac{\partial}{\partial \overline{z}_k}\right|_{z=z_0}\log \Bergman(z,z) \right\}.
$$
where $(h^{\bar{k} j})$ is the square root of $(g_B)^{-1}$.
\end{defn}

\begin{lem} \label{lem: dlogP = rep}
    For $z_0$ and $\xi$ in $\Omega$, 
    \begin{align*}
        |\partial \log P(z_0,\xi)|^2_{g_B} = |\operatorname{rep}_{z_0}(\xi)|^2  = \sum_{j=1}^{n} |b_j(\xi)|^2.
    \end{align*}
\end{lem}
\begin{proof}
    It follows from simple observation:
$$
\left.\frac{\partial}{\partial \overline{z}_k}\right|_{z=z_0}\log \frac{\Bergman(\xi,z)}{\Bergman(z,z)}
=
\overline{ \left.\frac{\partial}{\partial z_k}\right|_{z=z_0} \log P(z,\xi)}.
$$

\end{proof}

Therefore, the condition (\ref{gradient norm bound}) is equivalent to the existence of uniform bound for the images of the representative maps at all points $z_0\in\Omega$.

\begin{rmk}
    The image of the representative map is important in the theory of several complex variables.
    For instance, Lu Qi-Keng proved that if a bounded domain $\Omega$ admits a complete Bergman metric of constant holomorphic sectional curvature, then $\hbox{rep}_{z_0}(\Omega)=\mathbb{B}^n$ as a generalization of the Riemann mapping theorem (\cite{Lu1966onkahler}).
    However, in general, if the kernel is not zero-free, the image of the representative map need not be bounded, as in the case of an annulus (see \cite{skwarczynski1969invariant},\cite{yoo2017}).
\end{rmk}

\begin{lem} \label{lem: invariant of dlogP}
    Let $(\O_1, g_{B_1})$ and $(\O_2, g_{B_2})$ be bounded domains in $\CC^n$ with Bergman metrics, respectively. Let $\varphi: \O_1 \rightarrow \O_2$ be a biholomorphism. Then 
    \begin{align*}
        | \partial_z \log P_1(z, \xi)|^2_{g_{B_1}} = | \partial_z \log P_2(\varphi(z), \varphi(\xi))|^2_{g_{B_2}}
    \end{align*}
    for all $z, \xi \in \O_1$.
\end{lem}
\begin{proof}
    It follows from the transformation formula for the Bergman kernel and invariant of the Bergman metric by biholomorphisms:
    \begin{align*}
        \Bergman_1(z,\xi) = \Bergman_2(\varphi(z), \varphi(\xi)) \cdot \det(J\varphi(z)) \cdot \overline{\det(J\varphi(\xi))} \quad \text{ and } \quad
        \varphi^* g_{B_2} = g_{B_1}, 
    \end{align*}
    where $\det(J\varphi(\xi))$ is the determinant of complex Jacobian matrix of $\varphi$.
\end{proof}

\begin{cor} \label{cor: homogeneous Schwarz Lemma}
    Let $\O_1 \subset \CC^n$ be a bounded domain and $\O_2 \subset \CC^m$ be a bounded homogeneous domain. 
    Then, for any holomorphic map $f:\O_1 \rightarrow \O_2$, 
    $$
        f^*g_{B_2} \le C g_{B_1}
    $$
    for some constant $C>0$ that depends only on $\O_2$.
\end{cor}
\begin{proof}
    By Theorem~\ref{thm: Schwarz lemma for Bergman metric}, it is enough to show that 
    $$ \sup_{w, \zeta \in \O_2} | \partial_w \log P_2(w, \zeta)|^2_{g_{B_2}} $$
    is bounded. 
    
    Denote $F_{\O_2}(w, \zeta) := | \partial_w \log P_2(w, \zeta)|^2_{g_{B_2}}$.
    Let $\tau \colon \Omega_2 \to D$ be a biholomorphism from $\Omega_2$ onto a normal Siegel domain $D$ (\cite[Theorem 3.19 and Theorem 9.29]{xu2005theory}). 
    Then there exists a point $p \in D$ such that the representative image $\operatorname{rep}_p(D)$ of $D$ is bounded (see \cite[Theorem 4.7]{xu2005theory} for details).
    Since $D$ is a homogeneous domain, for a point $w \in \O_2$, there exists a biholomorphism $\varphi: D \rightarrow D$ such that $\varphi(\tau(w)) = p$.   
    Then, by Lemma~\ref{lem: invariant of dlogP},
    \begin{align*}
        F_{\O_2}(w, \zeta)  
        =& F_{\tau(\O_2)}( \tau(w), \tau(\zeta)) 
        = F_{D}(p, (\varphi \circ \tau)(\zeta)) \\
        =& |\hbox{rep}_p( (\varphi \circ \tau)(\zeta) )|^2 \le \hbox{diam}^2(\hbox{rep}_{p}(D)) < \infty,
    \end{align*}
    where the last equality comes from Lemma~\ref{lem: dlogP = rep}. 
    Therefore, $\sup_{w, \zeta \in \O_2} F_{\O_2}(w, \zeta)$ is bounded.
\end{proof}

\subsection{Constant gradient norm}
Up to this point, we have focused solely on the boundedness of the gradient norm; however, it is also important to find a `special' potential function of a K\"ahler metric whose gradient norm is {\it constant}.
In \cite{kai2007note}, Kai and Ohsawa proved that, for a bounded homogeneous domain, the potential function obtained from the Bergman kernel of a corresponding homogeneous Siegel domain has constant gradient norm.

Conversely, the existence of a potential function with constant gradient norm may be used to characterize the underlying domain (manifold).
For instance, Choi, Lee, and Seo recently established a characterization of the unit ball by the constant gradient norm of a K\"ahler-Einstein potential.

\begin{thm}[\cite{choi2023characterization}]
    Let $M$ be a simply connected complex manifold of dimension $n$ which covers a compact manifold and assume that $M$ admits a complete K\"ahler-Einstein metric $\omega$ with negative Ricci curvature $-K$. Suppose that there is a global potential function $\varphi: M \to \RR$ of $\omega$ such that
    $$
        \vert \partial \varphi \vert_{\omega}^2 = \frac{n+1}{K}.
    $$
    Then $M$ is biholomorphic to the unit ball in $\CC^n$.
\end{thm}

For a bounded homogeneous domain $\Omega$, it is well-known that the Bergman metric is K\"ahler-Einstein with negative Ricci curvature. In \cite{choi2025}, it is shown that the gradient norm of every (local) potential of the Bergman metric is $rc_\O$ whenever its gradient norm is constant, where $r$ and $c_{\O}$ are the rank and the genus of $\O$, respectively. For instance, one can find global potentials $\varphi_{\BB^n}$ and $\varphi_{\Delta^n}$ with $\vert \partial \varphi_{\BB^n} \vert^2_{g_B} \equiv n+1$ and $\vert \partial \varphi_{\Delta^n} \vert^2_{g_B} \equiv 2n$.
As a direct consequence, if $\O$ is a bounded symmetric domain in $\CC^2$, it admits a global potential with constant gradient norm. There are only two possible values, $3$ and $4$, which distinguish whether $\O$ is biholomorphic to $\BB^2$ or $\DD^2$.\\

Special attention has been paid to the problem of explicitly finding a potential function of a Kähler metric with constant gradient norm.
In this context, the density function $P(z,\xi)$ of the Bergman statistical model may serve as a candidate for such a potential function, provided that $\xi$ is chosen appropriately.
We expect that for some Shilov boundary point $\xi_0 \in\partial\Omega$, one can take
$$
\lim_{\xi\rightarrow \xi_0}(-\log P(\cdot,\xi))
$$
as a potential function of the Bergman metric with constant gradient norm, assuming that the off-diagonal Bergman kernel extends smoothly to the boundary.
The following two examples illustrate this phenomenon.

\begin{ex}
    Let $\BB^n$ be the unit ball in $\CC^n$. Then 
    $$\vert \partial \log P(z,\xi_0) \vert^2_{g_B} \equiv n+1$$ 
    whenever $\xi_0 \in \partial\BB^{n}$.
\end{ex}

\begin{proof}
    On $\BB^{n}$, the Bergman kernel $\Bergman$ and the Bergman metric $g_B = (g_{i\bar{j}})$ are given by
    $$
        \Bergman(z,\xi) = \frac{n!}{\pi^n} \frac{1}{(1 - z \cdot \bar \xi)^{n+1}}
    $$
    and
    $$
        g_{i \bar{j}} (z) = \frac{n+1}{(1-|z|^2)^2} ((1-|z|^2)\delta_{ij} + \bar z_{i} z_j),
    $$
    respectively. Let $(g^{\bar{j} k})$ be the inverse of $g_B$ with $g_{i \bar{j}} g^{\bar{j} k} = \delta_{ik}$;
    $$
        g^{\bar{j} k} (z) = \frac{1-|z|^2}{n+1} (\delta_{jk} - \bar z_{j} z_k).
    $$
    On the other hand, we have
    \begin{align*}
        \log P(z,\xi)    &= \log \left( \frac{|\Bergman(z,\xi)|^2}{\Bergman(z,z)} \right) \\
                    &= \log\left(\frac{n!}{\pi^n}\right) + (n+1)\left\{ \log (1-|z|^2) - \log (1-z\cdot \bar \xi)(1- \bar z \cdot \xi)\right\}.
    \end{align*}
    Its partial derivatives with respect to $z$ are given by
    $$
        \partial_i \log P(z,\xi) = -(n+1) \left( \frac{\bar z_i}{1-|z|^2} - \frac{\bar \xi_i}{1- z \cdot \bar \xi} \right).
    $$
    Therefore, 
    \begin{align*}
        \vert \partial \log P(z,\xi) \vert^2_{g_B} 
        =& g^{j \bar{k}}(z) (\partial_j \log P(z,\xi)) (\overline{\partial_k \log P(z,\xi)}) \\
        =& (n+1) |z|^2
        + (n+1) (1-|z|^2) \frac{|\xi|^2 - z \cdot \bar \xi - \bar z \cdot \xi + (z \cdot \bar \xi)(\bar z \cdot \xi)}{(1- z \cdot \bar \xi)(1- \bar z \cdot \xi)} 
    \end{align*}
    If $\xi_0$ lies on the boundary of $\BB^{n}$, then we have
    $$
        \frac{|\xi_0|^2 - z \cdot \bar \xi_0 - \bar z \cdot \xi_0 + (z \cdot \bar \xi_0)(\bar z \cdot \xi_0)}{(1- z \cdot \bar \xi_0)(1- \bar z \cdot \xi_0)}
        = 1,
    $$
    which yields $\vert \partial \log P(z,\xi_0) \vert^2_{g_B} = n+1$.
\end{proof}

\begin{ex}
    Let ${\Delta}^n$ be the unit polydisc in $\CC^n$.
    Then
    $$ \vert \partial \log P(z,\xi_0) \vert^2_{g_B} \equiv 2n $$
    whenever $\xi_0 \in (\partial {\Delta})^n$.
\end{ex}
\begin{proof}
    Since the Bergman kernel $\Bergman(x,\xi)$ is the product of $n$ copies of the Bergman kernel of $\Delta$ with $n$ independent variables, the Bergman metric is given by a diagonal matrix. This implies that the gradient norm of $\log P(z,\xi)$ is obtained from multiplying the one-dimensional result by $n$.
\end{proof}

\bibliographystyle{abbrv}
\bibliography{reference}

\end{document}